\documentclass{aptpub}

\usepackage[colorlinks = true,linkcolor = blue,urlcolor  = blue,citecolor = blue,anchorcolor = blue]{hyperref}
\authornames{V. H.~de~la~Pe\~na {\it et al.}} 
\usepackage{enumitem}
\usepackage{makeidx}
\usepackage{nccmath}
\usepackage{multirow,subfigure}
\usepackage{xr}

\usepackage{tikz}
\shorttitle{Moments of a Family of Self-Normalized Statistics} 

\newcommand{\R}{\mathbb{R}}

\newcommand{\lpar}{\left(}
\newcommand{\rpar}{\right)}
\newcommand{\lbr}{\left[}
\newcommand{\rbr}{\right]}
\newcommand{\pro}{\mathbb{P}}
\newcommand{\E}{\mathbb{E}}
\newcommand{\intzeroinf}{\int_0^{\infty}}

\newcommand{\widesim}[2][2]{\mathrel{\overset{#2}{\scalebox{#1}[1]{$\sim$}}}}

\newcommand\numberthis{\addtocounter{equation}{1}\tag{\theequation}}

     \newcommand{\VV}{\mathbb{V}}
     
\newcommand{\bx}{\mathbf{x}}   \newcommand{\bX}{\mathbf{X}}  
   \newcommand{\bY}{\mathbf{Y}}  
     
\newcommand{\bzero}{\mathbf{0}}

\graphicspath{{./figures/}}

\begin{document}

\title{A Scalable Formula for the Moments of a Family of Self-Normalized Statistics}

\authorone[Columbia University]{Victor H. de~la~Pe\~na}
\authortwo[Columbia University]{Heyuan Yao}
\authorone[Columbia University]{Haolin Zou}

\addressone{Department of Statistics, Columbia University, New York, NY 10027, USA.} 
\curaddtwo{Department of Industrial Engineering and Management Sciences, Northwestern University, Evanston, IL, 60208, USA}
\emailone{hz2574@columbia.edu} 

\begin{abstract}
    Following the student t-statistic, normalization has been a widely used method in statistic and other disciplines including economics, ecology and machine learning. We focus on statistics taking the form of a ratio over (some power of) the sample mean, the probabilistic features of which remain unknown. We develop a unified formula for the moments of these self-normalized statistics with non-negative observations, yielding closed-form expressions for several important cases. Moreover, the complexity of our formula doesn't scale with the sample size $n$. Our theoretical findings, supported by extensive numerical experiments, reveal novel insights into their bias and variance, and we propose a debiasing method illustrated with applications such as the odds ratio, Gini coefficient and squared coefficient of variation.
\end{abstract}



\section{Introduction}
\subsection{Literature Review}
Normalization is a ubiquitous technique that enables meaningful comparisons across datasets of different scales. Such quantities are often obtained by dividing the original quantity by a proxy of total amount, scale or variation, and thus expressed in the form of ratios and percentages, for example income per capita, alcohol by volume (ABV) and false discovery rates (FDR) in machine learning. In statistics, self-normalized statistics serve similar purposes to mitigate the effect of scale and variability in the observations, with a prominent example being the t-statistic (Student, 1908 \cite{student1908}, Gin\'e et al, 1997 \cite{gine1997student}), where the normalization is achieved by dividing the deviation of the sample mean by the standard deviation, thereby accounting for the intrinsic variability of the data. Broadly, normalization methods can be classified into two categories: those based on variability and those based on scale. The former includes the t-statistic, studentized residuals, and the Shapiro-Wilk statistic (Shapiro \& Wilk, 1965 \cite{shapiro-wilk1965}), while the latter encompasses measures such as the Gini coefficient (Gini, 1912 \cite{gini1912variabilita}) and the coefficient of variation (Pearson, 1898 \cite{pearson1898}), the latter of which is the focus of this paper. This approach is natural in contexts where the sample sum or sample average serve as the appropriate scaling factor. Despite the practical relevance of normalized statistics, their probabilistic properties, such as their bias and variance, remain insufficiently explored.

 Gin\'e et al (1997 \cite{gine1997student}) studied the condition under which the t-statistic is asymptotically normal. Besides the t-statistic, the study of moments of self-normalized statistics can date back to 1997 in an announcement of the Journal of the Academy of Science in Paris by Fuchs and Joffe (\cite{fuchs1997formule}), and later formalized in 2022 by Fuchs et al  \cite{fuchs2002expectation}, which provided a closed-form of 
\[
    \mathbb{E}\left[\frac{\sum_{i=1}^n X_i^2}{(\sum_{i=1}^n X_i)^2}\right]
\]
with i.i.d. observations $X_1, \cdots, X_n$.
This is also one of the earliest works that applied the identity 
\begin{equation}
    \label{equ: important identity}
    \frac{1}{x^\alpha} = \frac{1}{\Gamma(\alpha)} \intzeroinf \lambda^{\alpha-1}e^{-\lambda x} d \lambda
\end{equation}
which is also an important tool in this paper. The same trick was then applied in studying the higher moments and the limiting distribution of the aforementioned ratio ( \cite{albrecher2007asymptotic}, \cite{albrecher2008}), in Taylor's Law (\cite{brown2017taylor}), in the bias of odds ratio, relative risk and false discovery rate (\cite{pang2022bias}), in the unbiasedness of the Gini coefficient (\cite{banu2023}) and recently in the bias of Gini coefficient for Gamma mixture models (\cite{vila2025bias}).

The results mentioned above share a common characteristic: they all involve the study of ratios between a statistic and powers of the sample mean. However, there is not a unified and exact method to compute the moments of such statistics for all non-negative distributions. Our contribution fills the gap by developing a unified framework for deriving the moments of these self-normalized statistics, specifically those normalized by a power of the sample sum, for all non-negative distributions (both continuous and discrete). Our further analysis and numerical examples uncover novel patterns in the bias and variance of statistics such as the Gini coefficient and the squared coefficient of variation (SCV), and we propose a debiasing method that we illustrate through the Gini coefficient.

There has been another notable line of active study in self-normalization in stochastic processes, often in the context of sequential observations and online learning, where normalization is also carried out by a measure of variation (\cite{Pena-Pang-Exp2009, Pena-Lai-Shao-book, pena2007pseudo}), which have found extensive application in machine learning especially bandit problems and online-learning where observations arrive sequentially (\cite{Ramdas2021Time-Uniform, Abbasi2011Linear-Bandit, qin2014contextual}).

\subsection{Notations and conventions}
\label{ssec: notations}
We adopt conventional notations in probability and statistics. Specifically, we denote distribution functions by ordinary uppercase letters, such as $F$ and $G$, while scalar-valued statistics are represented by uppercase letters, including $T$, $V$, and $S$. Lowercase letters are used for probability density functions (e.g., $f, g$) as well as for other scalars and scalar-valued functions. Deterministic vectors are denoted by bold lowercase letters, such as $\bx \in \mathbb{R}^n$, whereas random vectors are represented by bold uppercase letters, such as $\bX$ and $\bY$. Parameters of distributions are indicated using lowercase Greek letters, for example, $\alpha, \beta,$ and $\gamma$.

The notation $\E_F$ and $\VV{\rm ar}_F$ refer to the expectation and variance (of the quantity after them) when the observations have distribution $F$.

We denote $\mathbb{N}_0$ as the set of non-negative integers and define $\mathbb{N} := \mathbb{N}_0 \backslash \{0\}$ as the set of positive integers. Similarly, we use $\mathbb{R}_+ := [0, \infty)$ to denote the set of non-negative real numbers. Throughout this work, we adopt the standard convention from (Lebesgue) measure theory that $0 \times \infty = 0$. For readers unfamiliar with this convention, it ensures that the measure (area) of a straight line in $\mathbb{R}^2$ is zero. In addition, the Gamma function $\Gamma(\alpha)$ is defined as:
$
    \label{equ: Gamma Function}
    \Gamma (\alpha)=\int_{0}^{\infty }t^{\alpha-1}e^{-t}\,dt
$
for any $\alpha>0$, which plays a central role in our theoretical framework.

\subsection{Structure of the paper}
The paper is organized as follows. Theoretical results are presented in Section \ref{sec: theoretical results}, where in Section \ref{ssec: main theorem} we present the main theorem for general ratio statistics and non-negative distributions, which is then applied to two specific statistics: the Gini coefficient (Section \ref{ssec: Consequence: Moments of Sample Gini Coefficient}) and the Squared Coefficient of Variation (Section \ref{ssec: Consequence: Moments of Sample Squared Coefficient of Variation}). 

In Section \ref{sec: applications} we provide further applications for the mean and variance of the two statistics above for selected distributions, as well as a novel debiasing method. To be more specific, in Section \ref{ssec: bias} we demonstrate the application of our formula in bias analysis: a novel method for proving the unbiasedness of the Gini coefficient for Gamma distribution (\cite{banu2023}) can be found in Section \ref{sssec: Proof Baydil's Theorem}, followed by the bias of $\hat{G}$ for Pareto distribution using numerical methods (Section \ref{sssec: pareto}), and then a novel debiasing method can be found in Section \ref{sssec: debiasing} with numerical experiments using Pareto distribution. In Section \ref{ssec: variance of Ghat under Gamma} we demonstrate the application of our method for calculating the variance of $\widehat{G}$ for Gamma distribution.

Finally, concluding remarks can be found in Section \ref{sec:discussion}. Additional numerical results for other distributions including Bernoulli, Lognormal, Negative Binomial, Inverse Gaussian and Poisson distributions can be found in the Appendix.

\section{Theoretical results}
\label{sec: theoretical results}

\subsection{Main theorem}
\label{ssec: main theorem}
In this section, we present the main theorem: a unified formula to calculate the moments of ratio statistics with the denominator being a power of the sample mean. Note that the method resembles that in Brown et al(2017 \cite{brown2017taylor}), but our formula allows for non-identical distributions with possible probability mass at zero ($\pro(X=0)>0$), and also arbitrary value for the ratio statistic when all observations are zero (where the ratio is not defined). To be more specific, consider a sample $\boldsymbol{X}:=(X_1,...,X_n)$ with $X_i$ being independent non-negative random variables with CDF $F_i(x)$ respectively. We are interested in ratio statistics with the following form:

\begin{align*}
    \label{eq:Self-normalized Statistics}\numberthis
    V(\bX)  := 
    \begin{cases}
        \frac{T(\bX)}{S_n^\alpha}&,\quad  \bX\neq \bzero\\
        r &,\quad  \bX = \bzero
    \end{cases}
\end{align*}
where 
\begin{itemize}[noitemsep, topsep=0pt]
    \item $T(\bX)$ is a statistic with finite expectation and $T(\bzero)=0$,
    \item $S_n:=\sum_{i=1}^n X_i$ is the sample sum, and 
    \item $\alpha>0,r>0$ are two constants.
\end{itemize}
\begin{remark}
    the value $r$ is introduced to ensure that the ratio remains well-defined when the denominator is zero. The choice of $r$ may depend on domain-specific knowledge or probabilistic considerations (see Section \ref{ssec: Bernoulli Law and the Explanation by Lorentz Curve}).
\end{remark}
\begin{remark}
    If a statistic $V=V(\bX)$ has the form \eqref{eq:Self-normalized Statistics}, its positive powers $V^k$ also has the same form with $T\leftarrow T^k, \alpha \leftarrow\alpha k$ and $r\leftarrow r^k$.
\end{remark}

The formulation in \eqref{eq:Self-normalized Statistics} encompasses many widely used statistics, including the Gini coefficient, the sample squared coefficient of variation (SCV), the Theil index, and the false discovery proportion (FDP), among others.

\begin{example}[Gini Coefficient]
    The (sample) Gini coefficient is a dimensionless (invariant in scale) measure of inequality:
    \[
        \widehat{G}(\bX) = \frac{1}{2(n-1)} \frac{ \sum\limits_{1\leq i\neq j \leq n}|X_i - X_j|}{S_n}.
    \]
\end{example}
\begin{example}[Squared coefficient of variation]
    The (sample) squared coefficient of variation (SCV) measures the spread of a sample:
    \[
        \widehat{c_V}^2 := \frac{n}{n-1}\frac{\sum\limits_{1\leq i< j \leq n}(X_i-X_j)^2}{S_n^2}.
    \]
\end{example}
\begin{example}[Theil Index]
    The Theil Index (also called Theil T index) is another measure of inequality:
    \[
        T_T(\bX) := \frac{\sum_{i=1}^n X_i \log(X_i/\bar{X}) }{S_n}
    \]
    where $\bar{X}$ is the sample mean.
    Note that this definition also coincides with the negative of Shannon's Diversity Index when the observations are counts of the occurrence of certain events.
\end{example}

Since the ratio \eqref{eq:Self-normalized Statistics} is not additive in general, calculating its expectation and higher moments usually involve $n$-layers of integrals. However, we provide a simplified formula of the moments of such kind of statistics. \footnote{Since our focus is not on the integrability itself, all expectations involved are assumed to be finite unless otherwise specified.}
Before presenting the theorem, several concepts need to be defined, which were also used in \cite{brown2017taylor} to study the Taylor's Law.
\begin{definition}[Laplace transform]
    For a univariate distribution with CDF $F$, its \textit{Laplace transform} is a function $L: [0,\infty)\to[0,1]$ defined as
\[
    \label{equ: Laplace Transform}
    L(\lambda) := \E_F\lpar e^{-\lambda X}\rpar  = \int_\R e^{-\lambda x} dF(x),\quad \lambda >0,\numberthis
\]
\end{definition}

\begin{definition}[Exponentially tilted family]
    For a univariate distribution with CDF $F$, the \textit{exponentially tilted distribution family} induced by $F$, or \textit{exponential tiltings} for short,  is a family of distributions $\{F^{(\lambda)}\}_{\lambda>0}$, defined by:
    \[
        dF^{(\lambda)}(x) = \frac{e^{-\lambda x}dF(x)}{L(\lambda)}.\label{eq:Titled}\numberthis
    \]
\end{definition}

Note that, when $F$ is continuous and has density $f$, $F^{(\lambda)}$ is also continuous and has density $f^{(\lambda)}(x) = f(x)e^{-\lambda x}/L(\lambda)$. With these concepts, we are ready to state the main theorem.

\begin{theorem}
    \label{thm: generalized Brown}
    Let $\bX = (X_1,...,X_n)$ be a random sample consisting of independent random variables $X_i \sim F_i(x)$, where $\{F_i(x)\}_{i=1}^n$ are CDFs on $[0,\infty)$ with Laplace transforms $L_i(\lambda)$. Let $V(\bX)$ have the form of \eqref{eq:Self-normalized Statistics},
    Then the expectation of $V(\bX)$ has the following formula:
    \[
        \label{eq:main theorem}\numberthis
        \E V(\bX)  = \frac{1}{\Gamma(\alpha)} \intzeroinf \lambda^{\alpha-1} \lbr\,\prod_{i=1}^n L_i(\lambda) \rbr \E_{F^{(\lambda)}}(T(\bX)) d\lambda + r\,\prod_{i=1}^n \pro(X_i = 0),
    \]
    where $F^{(\lambda)}:=\prod_{i=1}^n F_i^{(\lambda)}$ is the joint CDF of the exponentially tilted distributions.
\end{theorem}

\begin{proof}
    In this proof we let $F=F(\bx)=\prod_{i=1}^n F_i(x_i)$ be the joint distribution, and let $\E_F$ refer to taking expectation under the joint distribution $F$. We then have
    \begin{align}
        \E_F  V(\bX)
        & = \E_F \lbr r \boldsymbol{1}_{\{\bX=\bzero\}} \rbr + \E_F \lbr \frac{T(\bX)}{S_n^\alpha} \boldsymbol{1}_{\{\bX \neq \bzero\}} \rbr \nonumber \\
        & = r\,\prod_{i=1}^n \pro(X_i = 0) + \E_F\lbr \frac{T(\bX)}{S_n^\alpha} \boldsymbol{1}_{\{S_n >0\}} \rbr. \nonumber
    \end{align}

    It remains to show $ \E_F \lbr \frac{T(\bX)}{S_n^\alpha} \boldsymbol{1}_{\{S_n >0\}} \rbr = \frac{1}{\Gamma(\alpha)} \intzeroinf \lambda^{\alpha-1} \prod_{i=1}^n L_i(\lambda) \E_{F^{(\lambda)}}(T(\bX)) d\lambda$. 
    The main technique is the following gamma density trick: for $\alpha, x >0$ we have
    \[
        1=  \frac{x^\alpha}{\Gamma(\alpha)} \intzeroinf \lambda^{\alpha-1}e^{-\lambda x} d \lambda.
    \]
    because the right hand side is the density of a $Gamma(\alpha, x)$ distribution. By rearranging the terms we have
    \[
        \label{eq:Gamma Density Rewrite}
        \frac{1}{x^\alpha} = \frac{1}{\Gamma(\alpha)} \intzeroinf \lambda^{\alpha-1}e^{-\lambda x} d \lambda.
    \]
    Replacing $x$ by $S_n$ and multiplying both sides by $T(\bX)$ we have that, for $S_n>0$:
    \[
        \frac{T(\bX)}{S_n^\alpha}  = \frac{1}{\Gamma(\alpha)} \intzeroinf \lambda^{\alpha-1}T(\bX)e^{-\lambda S_n} d \lambda.
    \]
    Notice that the right hand side is $0$ when $S_n=0$, so we can rewrite it in a compact way to include the $S_n=0$ case:
    \[
        V(\bX) \boldsymbol{1}_{\{S_n>0\}}  = \frac{1}{\Gamma(\alpha)} \intzeroinf \lambda^{\alpha-1}T(\bX)e^{-\lambda S_n} d \lambda.
    \]
    Taking expectation to both sides and applying Fubini's theorem (because the integrands are non-negative) we have
    \begin{align*}
    \E_F &\lbr V(\bX) \boldsymbol{1}_{\{S_n >0\}} \rbr 
     = \E_F \lbr \frac{1}{\Gamma(\alpha)} \intzeroinf \lambda^{\alpha-1}T(\bX)e^{-\lambda S_n} \rbr \nonumber\\
     & = \frac{1}{{\Gamma(\alpha)}} \intzeroinf \E_F \lbr T(\bX) \lambda^{\alpha-1}e^{-\lambda S_n}  \rbr d \lambda \nonumber \\
    & = \frac{1}{{\Gamma(\alpha)}} \intzeroinf \int_{\R^n} T(x_1,...,x_n) \lambda^{\alpha-1}  e^{-\lambda (x_1+...+x_n)}   d F(x_1)...d F(x_n)d \lambda \\
    & = \frac{1}{{\Gamma(\alpha)}} \intzeroinf \int_{\R^n} 
    T(x_1,...,x_n) \lambda^{\alpha-1}  \prod_i^n L_i(\lambda)  d F^{(\lambda)}({x}_1)... d F^{(\lambda)}({x}_n) d \lambda\\
    & = \frac{1}{\Gamma(\alpha)} \intzeroinf \lambda^{\alpha-1} \prod_i^n L_i(\lambda) \E_{F^{(\lambda)}}(T(\bX)) d\lambda ,
    \end{align*}     
where the penultimate line uses the definition $dF^{(\lambda)}_i (x)= \frac{e^{- \lambda x}dF_i(x)}{L (\lambda)}$ $1\leq i \leq n$.
\end{proof}

The following Corollary is a direct application of Theorem \ref{thm: generalized Brown} to the case of i.i.d. observations.

\begin{corollary}
\label{col:generalized Brown_iid}    
    Under the same setting as in Theorem \ref{thm: generalized Brown} with the additional assumption that $F_1(x) = ...=F_n(x)\equiv F(x)$, we have
    \[
        \label{eq:main theorem_iid}
        \E_{F} V(\bX)  = \frac{1}{\Gamma(\alpha)} \intzeroinf \lambda^{\alpha-1} L^n(\lambda)\, \E_{F^{(\lambda)}}(T(\bX)) \,d\lambda \,+\, r\, \pro^n(X_1 = 0).\numberthis
    \]
\end{corollary}
The Proposition 1 of \cite{brown2017taylor} corresponds to the special case $r=0$ in the above corollary.
\begin{remark}
    Theorem \ref{thm: generalized Brown} and Corollary \ref{col:generalized Brown_iid} provide scalable formulae, the complexity of which do not depend on the sample size $n$. Note that this is usually not the case, as the expectation usually involves an n-dimensional integral unless the statistic itself has certain separability property, e.g, when it is a summation like a U-statistics, which is clearly not the case for a ratio statistic. But Theorem \ref{thm: generalized Brown} and Corollary \ref{col:generalized Brown_iid} simplifies the expectation to three components:
    \begin{itemize}[noitemsep, topsep=0pt]
        \item the the Laplace transform $L(\lambda)$,
        \item the expectation $\E_{F^{(\lambda)}}(T(\bX))$, and
        \item the final integral over $\lambda$,
    \end{itemize}  
    For many commonly used distributions, their Laplace transforms are either well-known or can be calculated easily. Additional properties of $T(\bX)$ can also facilitate the computation of $\E_{F^{(\lambda)}}(T(\bX))$,  e.g. when $T$ is a U-statistic:
    \[
        T(\bX)={n \choose k}^{-1}\sum_{1\leq i_1<\cdots<i_k\leq n} h(X_{i_1},\cdots,X_{i_k})
    \]
    for some kernel function $h: \R^k\to \R$, in which case we have
    \[
        \E_{F^{(\lambda)}}T(\bX) = \E_{F^{(\lambda)}}h(X_1,\cdots,X_k).
    \]
    and it is easier to calculate when $k$ is significantly smaller than n. When $k$ is fixed, the complexity of the formula doesn't scale with $n$ as it appears only as exponents of $L(\lambda)$ and $\pro(X_1=0)$.
    
    Lastly, for some distributions in the exponential family, the tilted distribution belongs to the original distribution family or a known family of distributions, with examples including Poisson, Gamma, and Binomial distributions etc, in which case $\E_{F^{(\lambda)}}T(\bX)$ has a closed form formula if $\E \,T(\bX)$ does. For example, when $X_i \widesim{i.i.d.} Poisson(\mu)$ with $dF_{\mu}(x)=\frac{\mu^x}{x!}e^{-\mu}$, the exponential tilted distribution is $dF^{(\lambda)}_{\mu}(x)\propto \frac{(\mu e^{-\lambda})^x}{x!}$ and turns out to be the $Poisson(\mu e^\lambda)$ distribution.
\end{remark}

In the following two subsections, we demonstrate the applicability of Theorem \ref{thm: generalized Brown} and Corollary \ref{col:generalized Brown_iid} by computing the bias and variance of the Gini coefficient and the squared coefficient of variation (SCV).

\subsection{Moments of the Gini coefficient}
\label{ssec: Consequence: Moments of Sample Gini Coefficient}

Introduced in 1912 (\cite{gini1912variabilita}), the Gini coefficient has been widely used as a dimensionless measure (invariant to the unit of measurement) of disparity in numerous fields including economics (\cite{chen2017}), demography (\cite{cohen2021}) and agriculture (\cite{sadras2004}), etc.
Among many equivalent definitions, we adopt the following version for the benefit of computation: 
\begin{equation}
    G = G(F) = \frac{\E_F |X_1-X_2|}{2\,\E_F X_1}.
    \label{Def of population Gini}
\end{equation}
where $X_1,X_2$ are two i.i.d. non-negative random variables from the same distribution $F$ of interest.

For a sample $\bX = (X_1 ,\cdots, X_n)$ drawn independently from $F$, the sample Gini coefficient can be defined as\footnote{We acknowledge an alternative definition that replaces the $n(n-1)$ factor with $n^2$, which corresponds to twice the area under the Lorentz curve (Woytinsky \cite{woytinsky1943earnings}). However, the difference in scaling constants is not substantial. From a statistical perspective, the version adopted in this work is more favorable and exhibits lower bias (see   Deltas \cite{deltas2003small}).
}

\[
\widehat{G}(\bX) = \frac{\frac{1}{n(n-1)} \sum\limits_{1\leq i\neq j \leq n}|X_i - X_j|}{2 \bar{X}_n}.
\label{Def of sample Gini HAT}\numberthis
\]
where $\bar{X}_n:= n^{-1}\sum_{i=1}^n X_i$ is the sample mean. This estimator is known to be consistent (Theorem A on pp190, \cite{serfling2009}). Moreover, the asymptotic distribution of $\widehat{G}$ for distributions with finite variance is known (Yitzhaki \& Schechtman, 2013 \cite{yitzhaki2013}). Fontanari et al (2018 \cite{fontanari2018}) further established the asymptotic distribution of $\widehat{G}$ for stable distributions with infinite variance.

For the small sample behavior of the Gini coefficient, the work  \cite{fontanari2018} suggests the presence of a downward bias of $\widehat{G}$ for heavy-tailed distributions. However, to evaluate the bias for finite samples, and more generally the moments $\E \widehat{G}^k$, one needs to perform an integral on $\R^n$ which is prohibitive for large $n$. Define the Gini Mean Difference (GMD) of $F$: 
\[
    GMD(F):= \int_{\R_+^2}|x_1-x_2|dF(x_1)dF(x_2) = \E_F |X_1 - X_2|
    \label{equ: GMD}\numberthis,
\]
then $G$ can be re-expressed in terms of $GMD$ as $G(F) = \frac{GMD(F)}{2\mu_F}$ where $\mu_F:=\E_F X_1$.

We now propose an exact formula for $\E_F(\widehat{G})$ and the ratio $R:=\frac{\E_F (\widehat{G})}{G}$ using Corollary \ref{col:generalized Brown_iid}, which reduces the n-layer integral to a triple integral. 

\begin{theorem}
\label{thm: Gini Expectation and Ratio}
    For a non-negative sample $\bX=(X_1,...,X_n) \widesim{i.i.d.} F$ with $n\geq 2$, we have
    \begin{enumerate}[label=(\roman*),ref=(\roman*)]
        \item 
        \begin{equation}
            \label{equ: Gini Expectation}
            \E_F \widehat{G} = \frac{n}{2} \intzeroinf GMD(F^{(\lambda)}) L^n(\lambda) \;  d \lambda + r\,\pro^n(X_1 = 0).
        \end{equation}
        where  $F^{(\lambda)}$ is the exponentially tilted distribution of $F$ defined in \eqref{eq:Titled}, $GMD(F^{(\lambda)})$ is the Gini mean difference of $F^{(\lambda)}$ and $L(\lambda)$ is the Laplace transform of $F$.
        \item Let $g(\lambda):=GMD(F^{(\lambda)})/GMD(F)$, then we have
        \begin{equation}
            \label{equ: R for Gini}
            R  :=\frac{\E_F\; \widehat{G}}{G}= n\mu_F \intzeroinf g(\lambda) L^n(\lambda) d\lambda + \frac{r\,\pro^n(X_1 = 0) }{G},
        \end{equation}
        where $\mu_F=\E_F(X_1)$. 
    \end{enumerate}
\end{theorem}

\begin{remark}
    Part (ii) of the theorem could facilitate computation when the function $g(\lambda)$ can be obtained without calculating $\E_F |X_1-X_2|$ first, for example for Gamma distribution, $g(\lambda) = (1+\lambda)^{-1}$ (see Section \ref{sssec: Proof Baydil's Theorem}).
\end{remark}

\begin{proof}
    By Theorem \ref{thm: generalized Brown}, we have, with $\alpha =1$, that
    \begin{align}
        \E_F \widehat{G} & =   \frac{1}{2(n-1)} \E_F \lbr \frac{\sum\limits_{1\leq i \neq j \leq n} |X_i - X_j|}{S_n}\rbr \nonumber \\
        & = \frac{1}{2(n-1)} \intzeroinf \E_{F^{(\lambda)}} \lpar \sum\limits_{1\leq i \neq j \leq n} |X_i - X_j| \rpar  L^n(\lambda)    d \lambda +  r\,\pro^n(X_1 = 0)\nonumber  \\
        & = \frac{n}{2} \intzeroinf GMD(F^{(\lambda)}) L^n(\lambda)    d \lambda +  r\,\pro^n(X_1 = 0) \label{equ: Gini Expectation proof 1},
    \end{align}
    where (\ref{equ: Gini Expectation proof 1}) uses the fact that \[\E \sum\limits_{1\leq i \neq j \leq n} |X_i - X_j| = n(n-1)\E_{F^{(\lambda)}} |X_1 - X_2| =n(n-1)GMD(F^{(\lambda)}).\] This concludes the proof of Part (i).
    
    For part (ii), by definition of $g(\lambda)$ we have $GMD(F^{(\lambda)})= g(\lambda) GMD(F)$, we can move $GMD(F)$ out of the integral:
    \begin{equation*}
        \E_F \widehat{G} = \frac{n}{2} GMD(F) \intzeroinf g(\lambda) L^n(\lambda) d\lambda + r\,\pro^n(X_1 = 0).
    \end{equation*}
    Hence
    \begin{equation*}
        R  = \frac{\E_F \widehat{G}}{G}  = n\E(X_1) \intzeroinf g(\lambda) L^n(\lambda) d\lambda + G^{-1}r\,\pro^n(X_1=0).
    \end{equation*}   
    This concludes the proof of part (ii).   
\end{proof}

In the next theorem we provide a formula for the second moment of $\widehat{G}$, and higher order moments can be calculated in a similar fashion. First we define the following quantities to simplify the expression.

\begin{definition}
\label{def: xi}
    Given $X_1, X_2, X_3,X_4 \widesim{i.i.d.} F$, we define the following terms:
    \begin{align}
        \xi_0(F) &= \E_F |X_1-X_2||X_3-X_4| =[GMD(F)]^2 \\
        \xi_1(F) &= \E_F |X_1-X_2||X_1-X_3| \\
        \xi_2(F) &= \E_F |X_1-X_2||X_1-X_2| = 2 \VV {\rm ar}(X_1).
    \end{align}
\end{definition}

\newpage

\begin{theorem}
\label{thm: Gini 2nd moment}
    With the same assumptions in Theorem \ref{thm: Gini Expectation and Ratio}, we have that
    \begin{enumerate}[label=(\roman*),ref=(\roman*)]
        \item
        \[
            \label{equ: 2nd moment of Sample Gini}
            \begin{aligned}
                E_F \widehat{G}^2&= \frac{1}{4(n-1)^2}  \intzeroinf \lambda   [ 2n(n-1) \xi_2(F^{(\lambda)}) + 4n(n-1)(n-2) \xi_1(F^{(\lambda)})   \\ & +n(n-1)(n-2)(n-3)\xi_0(F^{(\lambda)}) ]    L^n(\lambda)    d \lambda + r^2\pro^{n}(X_1 = 0).
            \end{aligned}
        \]
        \item if we define
        \[
        \label{equ: h_c in Gini 2nd moments}
            h_i(\lambda) = \xi_i(F^{(\lambda)})/\xi_i(F),        i = 0, 1, 2,
        \]
         we can rewrite the second moment as:
        \[
            \begin{aligned}
                \E_F \widehat{G}^2 =& \frac{n\xi_2(F)}{2(n-1)}  \intzeroinf \lambda h_2(\lambda)L^n(\lambda)  d\lambda + \frac{n(n-2)}{n-1}  \xi_1(F) \intzeroinf \lambda h_1(\lambda)L^n(\lambda)  d\lambda\\ 
                &+ \frac{n(n-2)(n-3)}{4(n-1)} \xi_0(F)\intzeroinf \lambda h_0(\lambda)L^n(\lambda)  d\lambda  +r^2\pro^{n}(X_1 = 0).
            \end{aligned}
        \]
    \end{enumerate}
\end{theorem}

\begin{proof}
    First notice that    
        \begin{align*}
            &\E_F ( \sum\limits_{1\leq i \neq j \leq n} |X_i - X_j| ) ^2  \\
            &= 2n(n-1)\xi_2(F)+ 4n(n-1)(n-2) \xi_1(F)+ n(n-1)(n-2)(n-3)\xi_0(F)  
            \label{Open square} \numberthis
        \end{align*}
        
    Then, by Theorem \ref{thm: generalized Brown} with $\alpha=2$, we can compute $\E(\widehat{G}^2)$ as
    \begin{align}
        \E_F \widehat{G}^2 & =  \frac{1}{4(n-1)^2} \E_F  \frac{\lpar \sum\limits_{1\leq i \neq j \leq n} |X_i - X_j|\rpar ^2}{S_n^2} \nonumber \\
        & = \frac{1}{4(n-1)^2}  \intzeroinf \lambda \E_{F^{(\lambda)}} \lpar \sum\limits_{1\leq i \neq j \leq n} |X_i-X_j| \rpar^2 L^n(\lambda)    d \lambda  \nonumber  +r^2\pro^{n}(X_1 = 0)\\
        & = \frac{1}{4(n-1)^2}  \intzeroinf \lambda   [ 2n(n-1) \xi_2(F^{(\lambda)}) + 4n(n-1)(n-2) \xi_1(F^{(\lambda)})  \nonumber  \\ 
        &                   + n(n-1)(n-2)(n-3)\xi_0(F^{(\lambda)}) ]    L^n(\lambda)    d \lambda +r^2\pro^{n}(X_1 = 0)  \nonumber\\
        & = \frac{n}{2(n-1)} \xi_2(F) \intzeroinf \lambda h_2(\lambda)L^n(\lambda)  d\lambda +  \frac{n(n-2)}{n-1} \xi_1(F) \intzeroinf \lambda h_1(\lambda)L^n(\lambda)  d\lambda  \nonumber  \\ 
        & + \frac{n(n-2)(n-3)}{4(n-1)} \xi_0(F)\intzeroinf \lambda h_0(\lambda)L^n(\lambda)  d\lambda+r^2\pro^{n}(X_1 = 0). \nonumber
    \end{align}
\end{proof}

\subsection{Moments of SCV}
\label{ssec: Consequence: Moments of Sample Squared Coefficient of Variation}
The squared coefficient of variation (SCV) is another measure of dispersion of a probability distribution, defined as
\[
c_V^2 = \frac{\VV {\rm ar}(X)}{\E^2(X)}
\]
for any random variable with $\E(X)\neq 0$. For non-negative i.i.d. random variables $\bX = (X_1,...,X_n)$, let $\hat{\sigma}^2$ denote the sample variance:
\begin{equation}
\label{equ: hat_sigma^2}
    \hat{\sigma}^2 =\frac{1}{n-1}\sum_{i=1}^n (X_i-\bar{X}_n)^2 = \frac{1}{n(n-1)}\sum\limits_{1\leq i < j \leq n}(X_i-X_j)^2 .
\end{equation}
A natural estimator of SCV, denoted by $\widehat{c_V}^2$, is defined as
\begin{equation}
\label{equ: hat c_V^2}
    \widehat{c_V}^2 := \frac{\hat{\sigma}^2}{\bar{X}_n^2} = \frac{n}{n-1}\frac{\sum\limits_{1\leq i< j \leq n}(X_i-X_j)^2}{S_n^2}.
\end{equation}
To avoid dividing-by-zero error, we make additional definition that $\widehat{c_V}^2=r$ when $X_1=\cdots=X_n=0$.
The theorem below provides exact formulae for $\E_F(\widehat{c_V}^2)$ and the ratio $R_V := \E_F(\widehat{c_V}^2) / c_V^2$:

\begin{theorem}
\label{thm: E(SCV)}
    Given non-negative random variables $X_1,...,X_n \widesim{i.i.d.} F$ with the corresponding Laplace transform $L(\lambda)$, we have
    \begin{enumerate}[label=(\roman*),ref=(\roman*)]
        \item 
        \begin{equation}
        \label{equ: E_widehat_c_V^2)}
            \E_F (\widehat{c_V}^2 ) =  r\,\pro^n (X_1=0)+n^2\intzeroinf \lambda \, \VV {\rm ar}_{F^{(\lambda)}}(X_1)\,L^n(\lambda)\, d \lambda 
        \end{equation}
        and 
        \begin{equation}
        \label{equ: hatR_V}
            R_V:= \frac{\E_F (\widehat{c_V}^2 )}{c_V^2} = \frac{r\,\pro^n (X_1=0)}{c_V^2}+ n^2 \frac{\E_F^2(X_1)}{\VV {\rm ar}_F(X_1)}\intzeroinf \lambda \, \VV {\rm ar}_{F^{(\lambda)}}(X_1) \, L^n(\lambda) \,d \lambda.
        \end{equation}

        \item  with $g(\lambda) := \VV {\rm ar}_{F^{(\lambda)}}(X_1)/\VV {\rm ar}(X_1)$, we have
        \begin{equation}
        \label{equ: simplified hatR_V}
            R_V = \frac{r\, \pro^n (X_1=0)}{c_V^2}+ n^2\, \E_F^2(X_1)\intzeroinf \lambda\, g(\lambda) \,L^n(\lambda) \,d \lambda.
        \end{equation}
    \end{enumerate}
\end{theorem}

\begin{proof}
    Here we only prove part (i) for brevity, and the proof of part (ii) is similar to that of Theorem \ref{thm: Gini Expectation and Ratio} (ii).
    
    Using Theorem \ref{thm: generalized Brown} with $T=\frac{n}{n-1}\sum\limits_{1\leq i< j \leq n}(X_i-X_j)^2$ and $\alpha =2$, we obtain that 

    \begin{align}
        \E_F \lpar \widehat{c_V}^2 \rpar & =  \intzeroinf \lambda \, L^n(\lambda) \E_{F^{(\lambda)}}(T(\bX)) d\lambda + r\,\pro^n(X_1 = 0) \nonumber \\
        & = \intzeroinf \lambda \, L^n(\lambda) \E_{F^{(\lambda)}}\lbr \frac{n}{n-1}\sum\limits_{1\leq i< j \leq n}(X_i-X_j)^2 \rbr d\lambda + r\,\pro^n(X_1 = 0) \nonumber \\
        & = \frac{n^2}{2} \intzeroinf \lambda \, L^n(\lambda) \E_{F^{(\lambda)}}\lbr (X_1-X_2)^2 \rbr d\lambda + r\,\pro^n(X_1 = 0) \nonumber \\
        & = n^2 \intzeroinf \lambda \, L^n(\lambda) \VV {\rm ar}_{F^{(\lambda)}}(X_1) d\lambda + r\,\pro^n(X_1 = 0) \nonumber.
    \end{align}
\end{proof}

\section{Applications and Numerical Results}
\label{sec: applications}
In this section, we present several applications of the theorems to specific families of distributions, including both analytical and numerical results. Notably, analytical results are rare, as the integral in \eqref{eq:main theorem} is generally not tractable and must be evaluated numerically. Additional numerical results are provided in the appendix.

\subsection{Bias analysis}
\label{ssec: bias}

\subsubsection{Unbiasedness of $\hat{G}$ for Gamma distribution}
\label{sssec: Proof Baydil's Theorem}

\indent

In this section we illustrate applying Theorem \ref{thm: Gini Expectation and Ratio} and \ref{thm: Gini 2nd moment} to $Gamma(\alpha,\beta)$ distribution defined by the following density:
\begin{equation}
\label{PDF of Gamma}
    f(y) = 
        \frac{\beta^{-\alpha}}{\Gamma(\alpha)}y^{\alpha-1}e^{-\frac{y}{\beta}},\quad y\geq 0.
\end{equation}
It is known that $\E X = \frac{\alpha}{\beta}$ and $L(\lambda) = (\lambda+1)^{-\alpha}$. Furthermore,   McDonald and Jensen (1979 \cite{mcdonald1979analysis}) provided the formula for the Gini Mean Difference and Gini coefficient for Gamma distribution:
\begin{equation}
    \label{eq: GMD of gamma}
    GMD(\alpha,\beta) := \E |X_1 -X_2| = \frac{2\Gamma(\alpha+\frac{1}{2})}{\sqrt{\pi}\Gamma(\alpha)\beta},
\end{equation} 
\begin{equation}
\label{popu Gini of Gamma}
    G(\alpha) = \frac{\Gamma(\alpha+\frac{1}{2})}{\sqrt{\pi}\Gamma(\alpha+1)}.
\end{equation}
Baydil et al (2025 \cite{banu2023}) showed that $\widehat{G}$ is an unbiased estimator for $G$ under Gamma distribution, which became one of the motivations of our study on the small sample bias of $\widehat{G}$ for other distributions. With the help of exponential tilting, we are able to provide an alternative and simple proof for the unbiasedness, based on Theorem \ref{thm: Gini Expectation and Ratio}:

\begin{corollary}[Unbiasedness of $\widehat{G}$ under Gamma distribution]
    \label{cor: Baydil's Theorem}
    For $X_i \widesim{i.i.d.} Gamma(\alpha,\beta)$, $1\leq i \leq n$, $\alpha,   \beta >0$, we have $\E \widehat{G}={G}$.
\end{corollary}

\begin{proof}
    Without loss of generality we assume $\beta=1$, since $\beta X_i \widesim{i.i.d.} Gamma(\alpha,1)$ for $X_i \widesim {i.i.d.} \Gamma(\alpha,\beta)$. 
    By \eqref{eq: GMD of gamma} we have $\E |X_1-X_2| = \frac{2\Gamma(\alpha+\frac{1}{2})}{\sqrt{\pi}\Gamma(\alpha)}$ and the exponential tiltings are given by 
    \[
        \frac{dF^{(\lambda)}(x)}{dx}\propto x^{\alpha-1}e^{-x}e^{-\lambda x}=x^{\alpha-1}e^{-(1+\lambda)x}
    \]
    so they follow $\Gamma(\alpha, 1+\lambda)$ distributions.
    Hence $\E_{F^{(\lambda)}}|X_1-X_2| =  (1+\lambda)^{-1} \E |X_1 - X_2|$, again by \eqref{eq: GMD of gamma}. According to part (ii) of Theorem \ref{thm: Gini Expectation and Ratio} and the fact that $L(\lambda) = (\lambda+1)^{-\alpha}$, we have
    \begin{align}
        R   =   \frac{\E\widehat{G}}{G} = \alpha \intzeroinf n \frac{1}{1+\lambda} L^n(\lambda) d\lambda   =   \alpha n \intzeroinf (\lambda + 1)^{-\alpha n -1}d\lambda   =  1.
    \end{align}
\end{proof}

\subsubsection{Pareto distribution}
\label{sssec: pareto}

Pareto (1898 \cite{pareto1898}) first observed that the income distribution can be approximated by a power law, later known as the Pareto distribution. Mitzenmacher (2003 \cite{mitzenmacher2003}) provided an interesting summary and bibliography on the application of Pareto distribution and log-normal distribution in economics, finance, computer science, biology, chemistry and astronomy. There are many variants of the Pareto distribution in the literature (e.g. \cite{arnold2008}), among which we use the two-parameter version:
\begin{equation}
    \label{equ: Pareto pdf}
    f(x) = \frac{\alpha x_m^\alpha}{x^{\alpha+1}}\boldsymbol{1}_{[x_m,\infty)},
\end{equation}
where $\alpha>1$ and $x_m>0$ are two parameters. If we assume that $\alpha>1$, the distribution has finite expectation. And because that $x_m>0$ is a scaling parameter (i.e., $k X_1\sim$ Pareto$(\alpha, k x_m)$), it is enough to study the Pareto$(\alpha,1)$ distribution. In this case, we have the mean $\mu = \E X_1 = \frac{\alpha}{\alpha-1}$, the Gini coefficient $G(\alpha) = \frac{1}{2\alpha-1}$, and the Laplace transform
\begin{equation}
    \label{equ: Laplace of Pareto}
    L(\lambda) = \E e^{-\lambda X} = \alpha \lambda^\alpha \Gamma(-\alpha, \lambda),   \forall \lambda>0,
\end{equation}
where $\Gamma(-\alpha, \lambda)$ is the upper incomplete Gamma function:
\begin{equation}
    \label{equ: incomplete Gamma function upper}
    \Gamma(-\alpha, \lambda) : =  \int\limits_\lambda^\infty x^{-\alpha-1}e^{-x}dx.
\end{equation}

Though Pareto$(\alpha,1)$ is in the exponential family, its corresponding exponential tilting is no longer a Pareto variable. However, we can still obtain the density function $f^{(\lambda)}(x)$, of $\Tilde{X}^{(\lambda)}$:
\begin{equation}
    \label{equ: pdf tilting Pareto}
    f^{(\lambda)}(x) =  \frac{\lambda^{-\alpha} x^{-\alpha-1}}{\Gamma(-\alpha, \lambda)} e^{-\lambda x},\quad x\geq 1,
\end{equation}
which coincides formally with a $Gamma(-\alpha,\lambda^{-1})$ density (usually for Gamma distribution both parameters should be positive).

By equation \eqref{equ: Gini Expectation} in Theorem \ref{thm: Gini Expectation and Ratio}, we have 
\begin{equation}
    \label{equ: expectation Gini Pareto}
    \E (\widehat{G}) = \frac{n}{2}\intzeroinf \int_1^\infty \int_1^\infty
    |x-y| \alpha^n \Gamma(-\alpha, \lambda)^{n-2} \lambda^{\alpha(n-2)} e^{-2\lambda(x+y)} (xy)^{-\alpha-1}
    dxdy) d\lambda
\end{equation}
And the ratio $R$ is then given by $R=(2\alpha-1)\E (\widehat{G})$.

Figure \ref{fig: Gini Pareto} visualizes the Gini coefficient, $\mathbb{E}(\widehat{G})$, and the ratio $R = \mathbb{E}(\widehat{G})/G$ for Pareto$(\alpha,1)$ distributed observations, considering various sample sizes $n$ and parameter $\alpha$. The figure clearly illustrates the downward bias of $\widehat{G}$, which becomes more pronounced when $\alpha$ and $n$ are smaller.

\begin{figure}[ht]
        \centering
	\includegraphics[width=14cm]{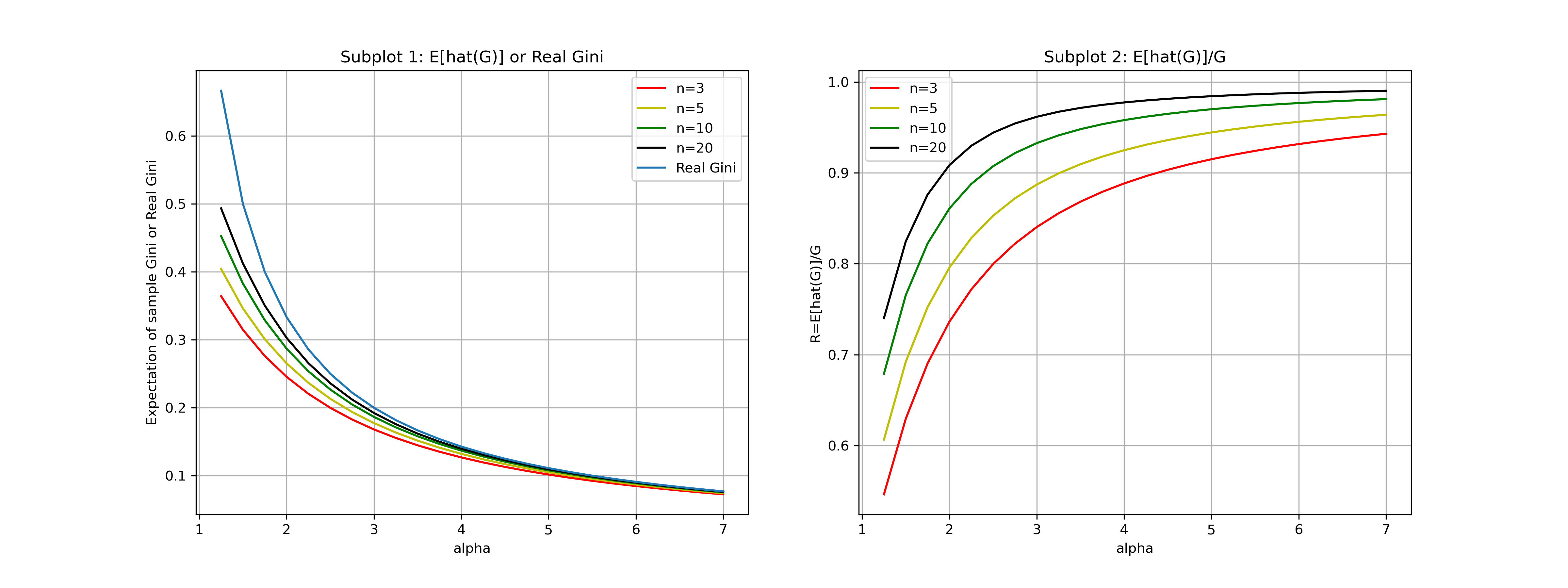}
	\caption{The Gini coefficient, $\E(\widehat{G})$ and $R$ of Pareto$(\alpha,1)$, with sample sizes $n=3,5,10$ and $20$. the x-axis represents the shape parameter $\alpha$.}
	\label{fig: Gini Pareto}
\end{figure}

\subsubsection{A novel debiasing method}
\label{sssec: debiasing}
As a direct application of the bias computed in the previous section, a natural approach is to use it for debiasing $\widehat{G}$. Specifically, we construct a new estimator, $\widehat{G}^{\text{debiased}}$, by subtracting an estimate of the bias from $\widehat{G}$, utilizing the bias formula provided in Theorem \ref{thm: Gini Expectation and Ratio}. This method can be readily extended to other statistics of the form \eqref{eq:Self-normalized Statistics}, provided that the underlying distribution is known.  

In this section, we focus on the Gini coefficient for illustrative purposes. A key challenge in implementing this approach is that the true population distribution is generally unknown. Consequently, a natural strategy is to estimate the distribution parameter—using, for example, the maximum likelihood estimator (MLE) or the method of moments estimator (MoM)—and substitute it into the bias formula before adjusting $\widehat{G}$. In this section, we demonstrate this debiasing procedure and compute the bias of such estimators for the Pareto distribution with density
\footnote{As we discussed earlier, $x_m$ is a scale parameter and doesn't affect the Gini coefficient, so we can set it to 1 without loss of generality.}
\[
    f(x) = \frac{\alpha }{x^{\alpha+1}}\boldsymbol{1}_{\{ x \geq 1 \}}.
\]
We compare the bias of the following estimators of Gini coefficient:
\begin{itemize}[noitemsep, topsep=0pt]
    \item classical sample Gini $\hat{G}$ defined in \eqref{Def of sample Gini HAT},
    \item debiased sample Gini using MLE: $\hat{G}^{MLE-debiased} = \hat{G} - bias(\hat{\alpha}^{MLE})$, where $bias(\alpha)$ is the bias of $\hat{G}$ calculated from Theorem \ref{thm: Gini Expectation and Ratio} numerically, and $\hat{\alpha}^{MLE}  := \frac{n}{\sum_{i=1}^n \log(X_i)}$ is the MLE of $\alpha$,
    \item Debiased sample Gini using MoM:
    $\hat{G}^{MoM-debiased} = \hat{G} - bias(\hat{\alpha}^{MoM})$, where $\hat{\alpha^{MoM}} = \frac{\bar{X}}{\bar{X}-1}$ is the method of moments (MoM) estimate of $\alpha$.
\end{itemize}
For comparison, we also compute the following two estimators of $G$ by inserting the MLE and MoM of $\alpha$ into the theoretical value $G(\alpha)=(2\alpha-1)^{-1}$: 
\begin{itemize}[noitemsep, topsep=0pt]
    \item plug-in estimator using MLE: $\hat{G}^{MLE}:= G(\alpha^{MLE}) = (2\alpha^{MLE} -1)^{-1}$,
    \item plug-in estimator using MoM: $\hat{G}^{MoM}:= G(\alpha^{MoM}) = (2\alpha^{MoM} -1)^{-1}$.
\end{itemize}

\begin{figure}[ht]
  \centering  
  \subfigure{\includegraphics[width=0.45\textwidth]{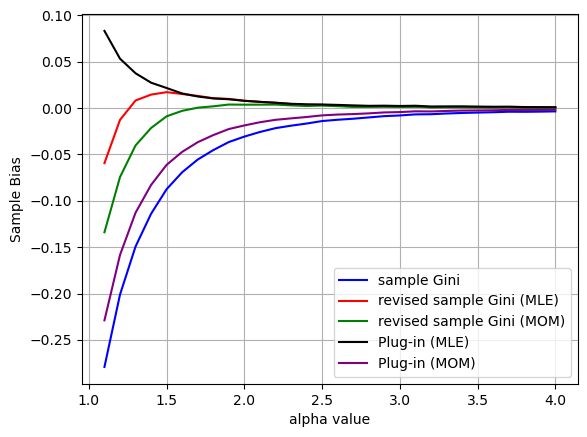}}
  \subfigure{\includegraphics[width=0.45\textwidth]{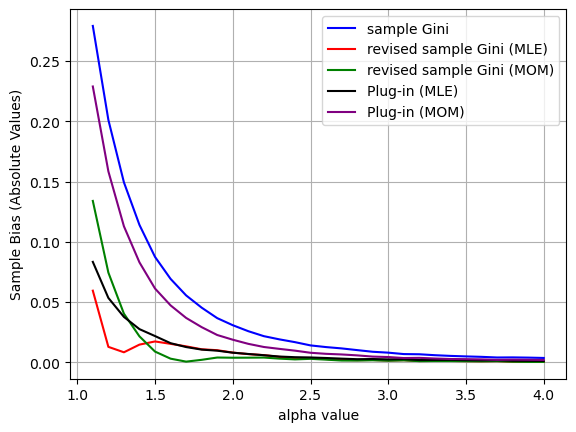}}
  \subfigure{\includegraphics[width=0.45\textwidth]{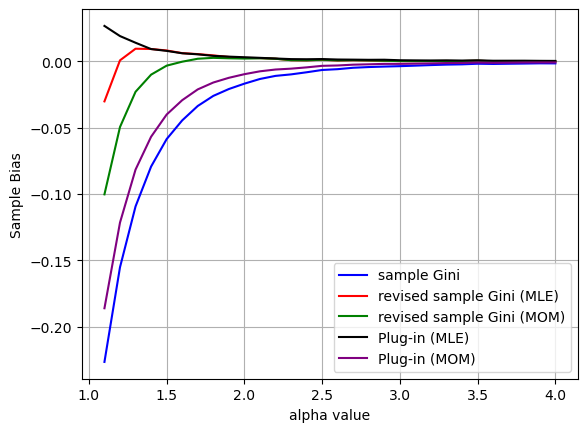}}
  \subfigure{\includegraphics[width=0.45\textwidth]{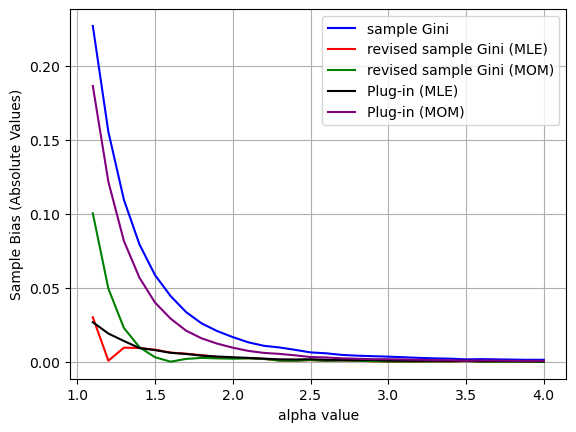}}
  \caption{Bias comparison of the five gini estimators for Pareto distribution, plotted against true value of $\alpha$. Top-left: bias for n=20 observations. Top-right: absolute bias for n=20. Bottom-left: bias for n=50. Bottom-right: absolute bias for n=50.}
  \label{Fig: Bias comparison of debiased estimators for Pareto distribution}
\end{figure}

Note that our debiasing method cannot eliminate the bias entirely, unless we know the true parameter and insert it into the bias function. Nonetheless, Figure \ref{Fig: Bias comparison of debiased estimators for Pareto distribution} compares the bias and its absolute value for the five aforementioned Gini estimators as a function of the true parameter $\alpha$, for sample sizes $n = 20$ and $n = 50$. Evidently, the plain Gini estimator, $\widehat{G}$, exhibits the highest bias in all cases, followed by the method-of-moments (MoM) plug-in estimator, $\widehat{G}^{\text{MoM}}$. The remaining three estimators—$\widehat{G}^{\text{MoM-debiased}}$, $\widehat{G}^{\text{MLE-debiased}}$, and the maximum likelihood plug-in estimator $\widehat{G}^{\text{MLE}}$—demonstrate similar bias performance. Among them, $\widehat{G}^{\text{MLE-debiased}}$ and $\widehat{G}^{\text{MLE}}$ perform slightly better when $\alpha$ is close to one, corresponding to extremely heavy-tailed distributions.

\subsection{Variance: Gamma distribution}
\label{ssec: variance of Ghat under Gamma}
For specific distributions, the variance of $\widehat{G}$  can be simplified based on Theorem \ref{thm: Gini 2nd moment}. In this section we provide an explicit formula for $\VV {\rm ar}(\widehat{G})$ for $Gamma(\alpha,\beta)$ distribution. The quantities in Definition \ref{def: xi} can be expressed as (WLOG assume $\beta=1$ since $\widehat{G}$ is scale invariant.)
\begin{align}
\xi_0 = \frac{4\Gamma^2(\alpha+\frac12)}{\pi\Gamma^2(\alpha)},\;\;\;\;\xi_1 := \E |X_1-X_2||X_1-X_3|,\;\;\;\;\xi_2 = 2 \alpha
\end{align}

\begin{corollary}
\label{SECOND MOMENT+VARIANCE}
    Given $X_i \widesim{i.i.d.} Gamma(\alpha,\beta)$, $1\leq i \leq n$, $\alpha,   \beta >0$, 
    \begin{equation}
    \label{equ: SECOND MOMENT}
        \E (\widehat{G}^2) = \frac{1}{(n-1)(\alpha n+1)} 
        + \frac{n-2}{\alpha(n-1)(\alpha n +1)}\xi_1 
        + \frac{(n-2)(n-3)}{\alpha(n-1)(\alpha n+1)}\frac{\Gamma^2(\alpha+\frac12)}{\pi \Gamma^2(\alpha)}
    \end{equation}
    and subsequently,
    \begin{equation}
    \label{equ: VARIANCE}
        \VV {\rm ar}(\widehat{G}) =
        \frac{1}{(n-1)(\alpha n+1)} 
        + \frac{(n-2)\,\xi_1 }{\alpha(n-1)(\alpha n +1)}
        - \frac{(1+4\alpha)n - (6\alpha+1)}{(n-1)(n\alpha +1)}\frac{\Gamma^2(\alpha+\frac12)}{\pi\alpha^2 \Gamma^2(\alpha)}
    \end{equation}
\end{corollary}

\begin{proof}
    Again without loss of generality we assume $\beta=1$, and $F^{(\lambda)}(x) \sim Gamma(\alpha,1+\lambda)\stackrel{L}{=}(1+\lambda)^{-1}Gamma(\alpha,1)$.
    Therefore we have that
    \begin{align*}
        h_k(\lambda) = \frac{\xi_k(F^{(\lambda)})}{\xi_k(F)}
        = \frac{1}{(1+\lambda)^2}\frac{\xi_k(F)}{\xi_k(F)}
        =\frac{1}{(1+\lambda)^2}, \quad k=0,1,2.
    \end{align*}
    Therefore, by Corollary \ref{thm: Gini 2nd moment}, we have that
    \begin{equation*}
        \E \widehat{G}^2 = \frac{1}{4(n-1)^2}\E \lpar \sum\limits_{1\leq i \neq j \leq n} |X_i - X_j| \rpar ^2 \intzeroinf \frac{\lambda}{(1+\lambda)^2}L^n(\lambda)  d\lambda,
    \end{equation*}
    where
    \begin{fleqn}[\parindent]
        \begin{align}
            \intzeroinf \frac{\lambda}{(1+\lambda)^2}L^n(\lambda)& = 
            \intzeroinf \lambda(1+\lambda)^{-\alpha n-2} d\lambda \nonumber \\
            & = \intzeroinf (1+\lambda)^{-\alpha n-1} d\lambda - \intzeroinf (1+\lambda)^{-\alpha n-2} d\lambda \nonumber \\
            & = \frac{1}{\alpha n} -  \frac{1}{\alpha n+1}  \nonumber = \frac{1}{(\alpha n)(\alpha n+1)} \label{step: will be used in SCV}
        \end{align}
    \end{fleqn}
    By combining the expectation (see equation \eqref{Open square}) and the integration, we prove the equation \eqref{equ: SECOND MOMENT}. And equation \eqref{equ: VARIANCE} follows as a result of the unbiasedness of $\widehat{G}$.
\end{proof}

\section{Concluding Remarks}
\label{sec:discussion}

In this paper, we proposed a unified and scalable formula (\ref{thm: generalized Brown}) for the moments of a class of normalized statistics for non-negative i.i.d. observations, which take the form of 
\[
    V(\bX):=\frac{T(\bX)}{\bar{X}^\alpha}.
\]
Our formula significantly simplifies the typically cumbersome computations of the expectation of such ratios, which often involve $n$-layered integrals, reducing them to a small number of integrals. This enables exact analyses of the bias and variance for a wide range of statistics. Notably, in cases where $T = \sum_{i,j} h(X_i, X_j)$—such as the Gini coefficient and the squared coefficient of variation (SCV)—our formula requires evaluating only four integrals, either numerically or in closed form.  

The key technique underlying our approach is the gamma density trick \eqref{equ: important identity}, which effectively handles the summation in the denominator. We demonstrated the utility of our formula by deriving explicit expressions for the expectation (\ref{thm: Gini Expectation and Ratio}) and variance (\ref{thm: Gini 2nd moment}) of the Gini coefficient, as well as the expectation of the SCV (\ref{thm: E(SCV)}). Furthermore, we provided numerical results for these formulae across several commonly used distributions, including the Gamma and negative binomial distributions.  

Based on these computations, we proposed a novel debiasing method (Section \ref{sssec: debiasing}) and demonstrated its superior performance in reducing bias for the Gini coefficient. Additional numerical experiments for a broader range of distributions—including Bernoulli, Pareto, log-normal, inverse Gaussian, and Poisson distributions—are provided in the appendices.

It is worth noting that our formula allows for the assignment of an artificial constant to the ratio when all observations are zero. While this scenario is highly unlikely in practice, it has implications for the theoretical analysis of moments. The specific choice of $r$ depends on domain-specific considerations and falls beyond the scope of our study.  

Furthermore, our main formula (\ref{thm: generalized Brown}) has the potential to extend beyond the applications presented in Sections \ref{sec: applications} and the Appendices, which are intended primarily for illustrative purposes. Additionally, future research may uncover more analytical results beyond the unbiasedness of $\widehat{G}$ established in Section \ref{sssec: Proof Baydil's Theorem}.

\acks
We wish to thank Professor Joel E. Cohen and Professor Rustam Ibragimov for inspiring and helpful discussions. We thank Victor K. de la Pena for helpful editing suggestions.  

\fund  There are no funding bodies to thank relating to this creation of this article.

\competing There were no competing interests to declare which arose during the preparation or publication process of this article.



%
%
%

\bibliographystyle{APT}
\footnotesize
\bibliography{ref}

\end{document}